\documentclass[11pt]{amsart}
\title[Spectral Simplicity of the Hodge Laplacian]{On Spectral Simplicity of the Hodge Laplacian and Curl Operator along Paths of Metrics}
\author{Willi Kepplinger}

\usepackage{amsmath}
\usepackage{amsthm}
\usepackage{amssymb}
\usepackage{graphicx}
\usepackage[backend=biber,style=alphabetic]{biblatex}
\usepackage{hyperref}
\usepackage[pagewise]{lineno}

\newtheorem{thm}{Theorem}[section]
\newtheorem{cor}[thm]{Corollary}

\newtheorem{lem}[thm]{Lemma}
\newtheorem{definition}[thm]{Definition}
\newtheorem{proposition}[thm]{Proposition}

\theoremstyle{remark}
\newtheorem*{rmk}{Remark}

\theoremstyle{plain}

\newtheoremstyle{note}
  {3pt}
  {3pt}
  {}
  {}
  {\itshape}
  {:}
  {.5em}
  {}

\theoremstyle{note}

\newtheoremstyle{citing}
  {3pt}
  {3pt}
  {\itshape}
  {}
  {\bfseries}
  {.}
  {.5em}
  {\thmnote{#3}}

\theoremstyle{citing}

\newtheoremstyle{break}
  {9pt}
  {9pt}
  {\itshape}
  {}
  {\bfseries}
  {.}
  {\newline}
  {}

\theoremstyle{break}

\swapnumbers
\theoremstyle{plain}

\let\lvert=|\let\rvert=|

\begin{document}
\maketitle

\begin{abstract}
We prove that the curl operator on closed oriented $3$-manifolds, i.e., the square root of the Hodge Laplacian on its coexact spectrum, generically has $1$-dimensional eigenspaces, even along $1$-parameter families of $\mathcal{C}^k$ Riemannian metrics, where $k\geq 2$. We show further that the Hodge Laplacian in dimension $3$ has two possible sources for nonsimple eigenspaces along generic $1$-parameter families of Riemannian metrics: either eigenvalues coming from positive and from negative eigenvalues of the curl operator cross, or an exact and a coexact eigenvalue cross. We provide examples for both of these phenomena. In order to prove our results, we generalize a method of Teytel \cite{Teytel1999}, allowing us to compute the meagre codimension of the set of Riemannian metrics for which the curl operator and the Hodge Laplacian have certain eigenvalue multiplicities. A consequence of our results is that while the simplicity of the spectrum of the Hodge Laplacian in dimension $3$ is a meagre codimension $1$ property with respect to the $\mathcal{C}^k$ topology as proven by Enciso and Peralta-Salas in \cite{Enciso2012}, it is not a meagre codimension $2$ property.
\end{abstract}

\section{Introduction and Statement of Results}
\label{section:introduction}

While the Hodge Laplacian on differential forms has not received as much attention as the Laplace-Beltrami operator (the Hodge Laplacian on functions), there has been a substantial amount of interest in this operator in recent years. The importance of its (coexact) spectrum in dimension $3$ has been recognized in the study of the geometry of closed $3$-manifolds (e.g., \cite{Lipnowski2018}), low dimensional topology (e.g., \cite{Lin2021}), and that of its eigenforms (or rather those of its square root, the curl operator) in the study of contact topology (e.g., \cite{Etnyre2000III},\cite{Etnyre2012}) and fluid dynamics (e.g., \cite{CarMirPerPres2019}). Properties of its spectrum have been studied in various contexts (e.g., \cite{ColetteJunya2012},\cite{Junya2003}) and are the subject of open problems \cite[Problem $8.24$]{CaivoveanuRassias2001}. 
\par

Generic properties of the eigenfunctions and eigenvalues of the Laplace-Beltrami operator were established in Uhlenbeck's landmark paper \cite{Uhlenbeck1976}. In particular, she showed that, given a closed $n$-dimensional manifold, the eigenvalues of the Laplace-Beltrami operator are generically simple, and that generic $1$-parameter families of Riemannian metrics connecting two Riemannian metrics for which the spectrum is simple have simple spectrum throughout. The natural question of whether similar statements could hold for the Hodge Laplacian was soon answered in the negative by Millman \cite{Millman1980} who proved that for closed even dimensional manifolds, multiplicities of nonzero eigenvalues of the Hodge Laplacian in the middle degree are always even. This means that a theorem in the generality of Uhlenbeck's cannot hold in this context.\par

At least in dimension $3$, however, Enciso and Peralta-Salas proved that the spectrum of the Hodge Laplacian is generically simple \cite{Enciso2012}. They did this in several steps, the first of which is to notice that the spectrum of the Hodge Laplacian in dimension $3$ is the same as that of its restriction to $1$-forms since the Hodge Laplacian $\Delta_g=\delta d+d\delta$ commutes with both the exterior derivative $d$ and the Hodge-star $\ast_g$. By the Hodge decomposition theorem \cite{Hodge1941} the eigen-1-forms to nonzero eigenvalues of the Hodge Laplacian split into exact and coexact ones, and so one naturally calls the spectrum of the Hodge Laplacian on the exact and coexact forms the exact and coexact spectrum, respectively. The exact spectrum coincides with the spectrum of the Laplace-Beltrami operator, and so by Uhlenbeck's result we already know that it is generically simple. Analyzing the coexact spectrum however proves much harder. Enciso and Peralta-Salas do this by studying the curl operator $\ast_g d$, which is the square root of the Hodge Laplacian on coexact forms. They follow the general strategy of Uhlenbeck (which is to apply a Sard-type theorem to the function $\Phi(g,u,\lambda)=(\ast_g d-\lambda)u$) except that significant analytical difficulties arise. Those include the PDE $\ast_g d u=\lambda u$ being vector valued and the curl operator not being elliptic (it has infinite dimensional kernel). After showing the generic simplicity of the spectrum of the curl operator, they still needed to break up the symmetric eigenvalues, that is the ones for which $\lambda_i=-\lambda_j$ for some $i,j$, before showing that the coexact spectrum of the Hodge Laplacian is generically simple. Finally, they examine a variation of the exact and coexact eigenvalues to conclude that even the full Hodge Laplacian has simple spectrum for a residual set in the Banach manifold of $\mathcal{C}^k$ Riemannian metrics with its natural topology.\par

In this article we want to complement and extend these results. Just like in the work mentioned above we say that a set is \textit{meagre} if it is the countable union of nowhere dense sets, and we say a set is \textit{residual} if its complement is meagre. A property is said to be \textit{generic} if it holds for a residual set. Note that all spaces appearing in this text are Baire spaces, and so residual sets are dense.

\begin{thm}
\label{theorem:main}
    Let $2\leq \ell<\infty$ and let $g_0$ and $g_1$ be two $\mathcal{C}^\ell$ Riemannian metrics. Then, for any $k\geq 1$, there exists a residual set in the space of paths $W^k:=\{w\in C^k([0,1],\mathcal{G}^\ell):w(0)=g_0, \,w(1)=g_1\}$ connecting these metrics such that the curl operator $\ast_{g(t)}d$ has simple spectrum for all $0<t<1$. 
\end{thm}
That is, the spectrum of the curl operator in dimension $3$ is simple along generic $1$-parameter families of Riemannian metrics. Moreover we prove that such a result is false for the Hodge Laplacian itself (unlike for the curl and the Laplace-Beltrami operator) in two different ways. On the one hand, we construct an example of two Riemannian metrics on $\mathbb{S}^3$ so that any path connecting them must have a crossing of eigenvalues, one coming from a positive and one coming from a negative eigenvalue of the associated curl operators, meaning that the Hodge Laplacian restricted to coexact $1$-forms does not have simple spectrum along $1$-parameter families of Riemannian metrics. On the other hand we also construct an example of two Riemannian metrics on $\mathbb{S}^3$ (different from the previous example) so that any path connecting them has a crossing of an exact and a coexact eigenvalue.
Calling those $1$-forms in the image of the spectral projector associated to the positive (negative) part of the spectrum of curl positive (negative), we get the following

\begin{thm}
    \label{theorem: Hodge Laplacian 1-parameter genericity}
    The Hodge Laplacian does not have simple spectrum along generic $1$-parameter families of Riemannian metrics. However, it does if one restricts to the closed $1$-forms, positive coexact $1$-forms, or negative coexact $1$-forms.
\end{thm}

This almost immediately implies 

\begin{cor}
\label{cor: meagre one but not two}
    The set of Riemannian metrics for which the Hodge Laplacian in dimension $3$ does not have simple nonzero spectrum is one of meagre codimension $1$ but not of meagre codimension $2$.
\end{cor}

which roughly means that simplicity of the spectrum of the Hodge Laplacian is a property that holds for the complement of a codimension $1$, but not a codimension $2$, set in the space of Riemannian metrics. The notion of meagre codimension is meant to capture the fact that some subset is not exactly a manifold of that codimension but that it it behaves like one under projections to subspaces of the parameter manifold so as to also include countable unions of manifolds of a given codimension. This will be made precise in subsection \ref{subsection:meagre codimension}.\par

Our approach is different from the method employed in both the Uhlenbeck and the Enciso and Peralta-Salas paper. Instead we adapt an idea that has been pioneered by de Verdière \cite{Verdière1988} and made more user friendly by Teytel \cite{Teytel1999}. The main insight is that, given a family of self-adjoint operators which differentiably depend on a parameter living in a separable Banach manifold, one can try to find local defining functions for the submanifold of parameter values which have double (or higher multiplicity) eigenvalues provided certain transversality conditions are satisfied. In this fortunate case, we can describe this non-simple subset as a set of meagre codimension $2$. Teytel's genericity criterion has found a number of applications, for example for proving generic eigenvalue properties of the Laplace-Neumann operator \cite{GomesMarrocos2019}.\par
An issue with this approach is that Teytel proved this theorem for a family of operators that are self-adjoint with respect to \emph{the same} inner product, whereas many classes of geometric operators we care about, such as the Hodge Laplacian, are self-adjoint with respect to the inner product induced by the Riemannian metric for which the operators are defined. It turns out that this difficulty is overcome rather easily as soon as the correct generalizations are chosen. Specifically, we obtain the following variation of Teytel's Theorem \cite[Theorem A]{Teytel1999}

\begin{thm}
    \label{theorem: meagre codimension criterion}
Let $A(q)$ be a family of operators whose resolvents $R_{A}(q)$ depend Fr\'echet-differentiably on a parameter $q$ that belongs to a separable Banach manifold $\mathcal{X}$, each densely defined on the same real Banach space $\mathcal{H}$. Furthermore every $A(q)$ is self-adjoint with respect to a differentiable family of inner products $\langle-,-\rangle_q$ defined on $\mathcal{H}$, each inducing norms equivalent to that of $\mathcal{H}$. Assume that the spectrum of each operator $A(q)$ is discrete, of finite multiplicity, and with no finite accumulation points. Assume also that the family $A(q)$ satisfies SAH$2$. Then the set of all $q$ such that $A(q)$ has a repeated eigenvalue has meagre codimension $2$ in $\mathcal{X}$.
\end{thm}

Here the condition SAH$2$ derives from the ``strong Arnold hypothesis'' introduced by de Verdière in \cite{Verdière1988} and is essentially the transversality condition needed in order to conclude that we find the local defining functions mentioned above. This will be made precise in section \ref{subsection:the genericity result}, definition \ref{def: SAH2}. This approach is chosen since it would be unclear how to use the Uhlenbeck and Enciso and Peralta-Salas method in order to prove Theorem \ref{theorem:main}. Apart from that, however, it gives us stronger statements than we would get by using the former ideas. In particular, the de Verdière-Teytel method allows one to actually determine the (meagre) codimension of the set of Riemannian metric for which some family of geometric operators has eigenvalues of a given multiplicity, and not just prove that it is meagre. Further, Uhlenbeck's method is only powerful enough to prove that generic $1$-parameter families $g_t$ of Riemannian metrics connecting metrics $g_0$ and $g_1$ with simple spectrum have simple spectrum for all $t$, but no such restrictions on $g_0$ and $g_1$ exist if one uses Theorem \ref{theorem: meagre codimension criterion}.\par

We wish to remark that while the curl operator is defined on $k$ forms on manifolds of dimension $2k+1$, it is not true that the results presented here readily generalize to higher odd dimensions. Indeed it was proven by Gier and Hislop \cite{GierHislop2016} that the curl operator in dimension $5$ has generically simple spectrum, but because of the skew symmetry of this operator in dimension $4k+1$, this implies that its square, the Hodge Laplacian restricted to the coexact $2$-forms in dimension $5$, generically has $2$-dimensional eigenspaces. The authors of \cite{GierHislop2016} conjecture that the coexact spectrum of the Hodge Laplacian on $k$ forms in dimension $2k+1$ is generically $2$ dimensional when $k$ is even and generically $1$ dimensional when $k$ is odd.\par
Even if one aims to prove the generic simplicity of the spectrum of the curl operator in higher dimensions using the techniques presented here one will have to verify the SAH$2$-condition in a different way as we really do use the fact that the curl operator in dimension $3$ is defined on $1$-forms.\par

This papers is organized as follows: we will first recall the notion of meagre codimension as introduced by Teytel in \cite{Teytel1999} in subsection \ref{subsection:meagre codimension}, then review the setup for the de Verdière-Teytel method and introduce the necessary modifications in \ref{subsection:the genericity result}. Following this we apply Theorem \ref{theorem: meagre codimension criterion} to the curl operator, proving that simplicity of the coexact spectrum holds in the complement of a meagre codimension $2$ subset. We then use the fact that this approach immediately lends itself to the application of the study of $k$-parameter families of operators in order to prove Theorem \ref{theorem:main}.\\
Finally we construct counterexamples to simplicity of the Hodge Laplacian along generic $1$-parameter families of Riemannian metrics mentioned above and thereby prove Theorem \ref{theorem: Hodge Laplacian 1-parameter genericity} in subsection \ref{subsection: the one parameter family arguments}.

\section*{Acknowledgements}
I want to thank my advisor Vera V\'ertesi as well as Michael Eichmair for their constant support and helpful mentoring as well as the Vienna School of Mathematics for providing a stable and pleasant research environment. I am also grateful to Daniel Peralta-Salas, Junya Takahashi, and John Etnyre for fruitful discussion and the provision of essential references. Special thanks go to my good friend Josef Greilhuber for a particularly careful reading of an earlier version of the present paper and many helpful remarks. Finally the author would like to thank the anonymous referees for helpful comments and suggestions.\\
This research was funded in part by the Austrian Science Fund (FWF) [10.55776/P34318] and [10.55776/Y963]. For open access purposes, the author has applied a CC BY public copyright license to any author-accepted manuscript version arising from this submission.

\section{Preliminaries}
\label{section:Preliminaries}

\subsection{Meagre Codimension}
\label{subsection:meagre codimension}
In this section, we recall the definition of meagre codimension and some of its basic properties from \cite{Teytel1999}. We also prove a simple technical lemma that we will need in Section \ref{subsection: the one parameter family arguments}. 

\begin{definition}{\cite{Teytel1999}}
     Let $X$ be a Banach space and $Z\subset X$ be a hyperplane of codimension $n$. Then a differentiable map $\pi:X\to Z$ is called a nonlinear projection if, for every $x\in X$, $\pi^\prime_x:X\to Z$ is surjection and has kernel of dimension $n$.
\end{definition}
   
Using these nonlinear projections allows Teytel to quantify more exactly how small a set is. 

\begin{definition}{\cite{Teytel1999}}
        A subset $Y\subset X$ is said to be of meagre codimension $n$ if $\pi(Y)\subset Z$ is meagre in $Z$ for every hyperplane $Z$ of codimension $n-1$ and every nonlinear projection $\pi:X\to Z$.
\end{definition}

There is a natural way of extending this definition from Banach spaces to Banach manifolds.

\begin{definition}{\cite{Teytel1999}}
    Given a Banach manifold $M$, we say that a set $Y\subset M$ has meagre codimension $n$ if for every chart $(\phi,U)$ of $M$, $\phi(U\cap Y)$ has meagre codimension $n$ in $\phi(U)$. 
\end{definition}

Examples of subsets of meagre codimension $n$ of a Banach space $X$ include codimension $n$ hyperplanes and smooth codimension $n$ submanifolds of $X$. Clearly, countable unions of sets of meagre codimension $n$ are again sets of meagre codimension $n$, and sets of meagre codimension $n$ are also sets of meagre codimension $m$ for all $m\leq n$.\par

We will need the following Lemmas proven by Teytel:
\begin{lem}{\cite{Teytel1999}}
\label{lemma:meagre}
A subset $Y$ of a separable Banach space is of meagre codimension $1$ iff it is meagre.
\end{lem}
The proof carries over almost verbatim to subsets of separable Banach manifolds. 
\begin{lem}
\label{lem:codimension intersection}{\cite{Teytel1999}}
Let $Y$ be a set of meagre codimension $n$ and let $M$ be a submanifold of $X$ of codimension $k< n$. Then $Y\cap M$ is a set of meagre codimension $n-k$ in $M$.
\end{lem}
\begin{lem}{\cite{Teytel1999}}
\label{lem:meagre codim local reduction}
Let $X$ be a separable Banach manifold and $Y$ be a subset of $X$. Suppose that for every $q\in Y$ there exists a neighbourhood $U_q$ of $q$ such that $Y\cap U_q$ is of meagre codimension $n$ in $U_q$. Then $Y$ is of meagre codimension $n$ in $X$.
\end{lem}
Finally we prove a Lemma we will use for the $1$-parameter arguments.

\begin{lem}
\label{lemma:codimension}
Let $X$ be a Banach manifold, and let $Y\subset X\times \mathbb{R}^k$ be a set of meagre codimension $n$ with $n\geq k$, and let $\pi:X\times \mathbb{R}^k\to X$ be a nonlinear projection. Then $\pi(Y)\subset X$ is a set of meagre codimension $n-k$.
\end{lem}
\begin{proof}
Let $(U_i,\phi_i)$ be a collection of charts covering $X$, and $(U_i\times \mathbb{R}^k,\psi_i)$ be the associated product charts. Then $\pi_{\psi_i}$, the chart representative of the nonlinear projection, is a nonlinear projection from $\psi_i(U_i\times \mathbb{R}^k)$ to $\psi_i(U_i\times \{0\})=\phi_i(U_i)\times \{0\}$. By the above lemmas, $\pi_{\psi_i} (\psi_i(Y\cap U_i\times \mathbb{R}^k))\subset \phi_i (U_i)$ is a set of codimension $n-k$. Since this is true for any choice of charts $(U_i,\phi_i)$ we are done.
\end{proof}

\subsection{The Genericity Result}
\label{subsection:the genericity result}
We adapt the de Verdière-Teytel construction. Let the parameter space $\mathcal{X}$ be a separable Banach manifold. Let $(\mathcal{H},\langle -,-\rangle_q)$ be a family of real Hilbert spaces, Fr\'echet-differentiably parametrized by $q\in\mathcal{X}$ with the same underlying real Banach space $\mathcal{H}$, i.e. all inner products induce equivalent norms. Let furthermore $A_q$ be linear operators satisfying:

\begin{itemize}
    \item $A_q$ is self-adjoint with respect to the inner product $\langle-,-\rangle_q$ for all $q\in\mathcal{X}$
    \item the spectrum $\sigma(A_q)$ of $A_q$ consists of countably many discrete eigenvalues, all of which have finite multiplicity for all $q\in\mathcal{X}$
    \item The resolvents $R_{A_q}(\mu)=(\mu-A_q)^{-1}$, $\mu\in\rho(A_q)$, of $A(q)$ depend Fr\'echet-differentiably on $q$ 
\end{itemize}

\begin{rmk}
    Since self adjoint operators are closed, the resolvents $R_{A_q}$ are bounded operators, so it is clear what Fr\'echet-differentiability means in this context. Teytel required the operators $A_q$ to depend Fr\'echet-differentiably on $q$ with respect to the graph norm, but we find this formulation to be slightly cleaner. Note that, unlike Teytel, we do not need to require the domains $D(A_q)$ to coincide for all $q\in \mathcal{X}$ since the resolvents are defined on all of $H$.
\end{rmk}

Fix some $q_0\in\mathcal{X}$ and an eigenvalue $\lambda$ of $A_{q_0}$ with multiplicity $m$. The general strategy will now be to find defining functions for the submanifold of parameter values $q$ close to $q_0$ for which the part of the spectrum of $A_q$ near $\lambda$ is not simple.\par
Since the spectrum of $A_{q_0}$ is discrete there exists an $\epsilon >0$ such that $\lambda$ is the only eigenvalue of $A_{q_0}$ in $(\lambda-\epsilon,\lambda+\epsilon)$. We now consider an open neighbourhood $\mathcal{U}(q_0)\subset \mathcal{X}$ of $q_0$ such that the spectrum of $A_q$ in $(\lambda-\epsilon,\lambda+\epsilon)$ consists of $n$ eigenvalues whose multiplicities sum to $m$. Such an open neighbourhood exists because $\sigma(A_q)$ depends continuously on $q$ \cite[pp. 372-373]{RieszNagy1955}. Note that the previous reference only deals with the case of families of operators that are self-adjoint with respect to the same inner product, but the same proof goes through in our case. Continuous dependence on $q$ also follows from the min-max characterization of eigenvalues of selfadjoint operators \cite[Theorem 4.10]{GeraldTeschl2014}.\par
We denote the associated sum of eigenspaces by $E(q)$ and define the spectral projection $P:\mathcal{U}(q_0)\times \mathcal{H}\to E(q)$ by
\begin{align*}
    P(q)=\frac{1}{2\pi i}\int_\gamma R_{A_q}(\gamma)\,d\gamma
\end{align*}
Where $\gamma$ is the simple closed curve $\lambda+\epsilon\,e^{it}$ in the complex plane. Note that, while all the $(\mathcal{H},\langle-,-\rangle_{q})$ are real Hilbert spaces, we can first complexify, then apply the operator $P(q)$, and then restrict to the real elements of the complexification. Furthermore, the projection $P(q_0)$ is an isomorphism from $E(q)$ to $E(q_0)$ for $q$ close enough to $q_0$. The fact that $P$ depends differentiably on the parameter $q$ follows from an application of dominated convergence.\par
Using this projection we define the map 
\begin{align*}
    S(q)=P(q)\circ P(q_0):E(q_0)\to E(q) 
\end{align*}
and the local defining function $f:\mathcal{U}(q_0)\subset\mathcal{X}\to GL(E(q_0))$ by
\begin{align*}
    f(q)=S(q)^{-1}R_{A}(q)S(q)
\end{align*}
Here we do not specify the argument of the resolvent since it does not matter and denote it by $R_{A}(q)=R_{A(q)}$. We do however want the argument to be real since we are working in a real Banach space. One can find a real argument $\mu$ so that the resolvent $R_{A(q)}(\mu)$ is defined for all $q\in \mathcal{U}(q_0)$ provided $\mathcal{U}(q_0)$ is chosen small enough.\par
Note that the preimage of the submanifold $\mathbb{R}\cdot Id\subset GL(E(q_0))$ under $f$ is precisely the set of parameter values for which there is a single eigenvalue in the interval $(\lambda-\epsilon,\lambda+\epsilon)$. This eigenvalue is necessarily of multiplicity $m$.\par
To see that this set is a submanifold we have to compute the derivative of $f$ at $q_0$. Using that $S(q_0)=Id$ one sees that this is given by
\begin{align*}
    f^{\prime}(q_0)=[S^{\prime}(q_0),R_{A}(q_0)]+R_{A}^{\prime}(q_0)
\end{align*}
Let $\{v_i\}_{1\leq i\leq m}$ be an orthonormal eigenbasis of $(E(q_0),\langle-,-\rangle_{q_0})$ for the operator $A_{q_0}$. Given an element $h\in T_{q_0}\mathcal{X}$ we can now express the endomorphism $f^{\prime}(q_0)[h]$ in this basis and see that $\langle [S^{\prime}(q_0),R_{A}(q_0)][h] v_i,v_j\rangle_{q_0}=0$ by the fact that $R_{A}(q_0)$ is symmetric with respect to $\langle-,-\rangle_{q_0}$. Thus
\begin{align*}
    f^\prime_{ij}(h):=\langle f^{\prime}(q_0)[h] v_i,v_j\rangle_{q_0}=\langle R_{A}^{\prime}(q_0)[h] v_i,v_j\rangle_{q_0}
\end{align*}

We can now define the condition SAH$2$

\begin{definition}{\cite{Teytel1999}}
\label{def: SAH2}
Let $\mathcal{H}$ be a real Hilbert space, and $\lambda$ be an eigenvalue of $A(q)$ of multiplicity $n\geq 2$. We say that the family $A(q)$ satisfies the condition SAH$2$ if there exist two orthonormal eigenvectors $v_1$ and $v_2$ of eigenvalue $\lambda$ such that the linear functionals $f^\prime_{11}-f^\prime_{12}$ and $f^\prime_{12}$ are linearly independent.
\end{definition}

\begin{rmk}
The spectrum of an operator $A$ is simple iff the spectrum of its resolvent $R_{A}$ is simple. Moreover, whenever the derivative of $A_q$ exists in some suitable sense, $\langle R_{A}^{\prime}(q_0)[h] v_i,v_j\rangle_{q_0}$ and $\langle A^{\prime}(q_0)[h] v_i,v_j\rangle_{q_0}$ only differ by a constant factor so in that case $R_{A}(q)$ satisfies the condition SAH$2$ iff $A(q)$ does. 
\end{rmk}

\begin{rmk}
  In Teytel's setting, $A^\prime (q_0)$ always maps to $\mathcal{L}(E(q_0))$, the space of symmetric endomorphisms on $E(q_0)$. This means that if the multiplicity of $\lambda$ is two, the linear functions appearing in condition SAH$2$ precisely span all directions which are transverse to $\mathbb{R}\cdot Id$. In our setting there is one more direction that $A^\prime(q)$ could potentially map to, however for all examples known to us, $A^\prime (q_0)$ also maps to $\mathcal{L}(E(q_0))$.
\end{rmk}

\begin{proof}[Proof of Theorem \ref{theorem: meagre codimension criterion}]
The proof can be copied almost verbatim from the one provided by Teytel in \cite[Chapter $3$]{Teytel1999}, the only difference being that $f$ maps to $GL(E(q_0))$ as opposed to $\mathcal{L}(E(q_0))$. For the sake of completeness we outline the most important steps leading up to that point. The idea is to filter the set $\mathcal{D}\subset X$ of metrics for which the operator $A$ does not have simple spectrum by the multiplicity of the non-simple eigenvalues, and then show that locally $\mathcal{D}$ has meagre codimension 2, which then allows us to conclude the argument by appealing to Lemma \ref{lem:meagre codim local reduction}.\\
Given a finite open interval $I$, we define the set $\mathcal{D}_{n,I}\subset X$ as the collection of parameters $q\in X$ such that $A_q$ has an eigenvalue $\lambda$ of multiplicity $n$ but no eigenvalue of higher multiplicity contained in $I$. Then
\begin{align*}
    \mathcal{D}=\bigcup_{N\in \mathbb{N}}\bigcup_{n>1} \mathcal{D}_{n,(-N,N)}
\end{align*}
and so if we show that for all $n>1$ and all finite intervals $I$, $\mathcal{D}_{n.I}$ has meagre codimension $2$, we are done since the countable union of sets of meagre codimension $2$ is again a set of meagre codimension $2$. Fixing $q_0\in \mathcal{D}$ we note that any finite interval $I$ only contains finitely many eigenvalues of $A_{q_0}$. Restricting to those eigenvalues $\lambda_1,\dots, \lambda_s$ with multiplicity $m(\lambda_i)\leq n$ and using the fact that eigenvalues move continuously with the parameter we see that there exists a small neighbourhood $U_{q_0}$ of $q_0$ so that $D_{n,I}\cap U_{q_0}=\bigcup_{i=1}^{s}D_{n,(a_i,b_i)}\cap U_{q_0}$, where the sum of multiplicities of eigenvalues in $(a_i,b_i)$ is $m(\lambda_i)$ for all $A_q$ with $q$ in $U_{q_0}$. It thus suffices to show that $U_{q_0}\cap\mathcal{D}_{n,(a,b)}$ with $(a,b)$ containing eigenvalues of total multiplicity $n$ has meagre codimension $2$ in $U_{q_0}$.\\
When $n=2$ then $A_q$ satisfying SAH$2$ gives two linearly independent defining equations, thus implying that $\mathcal{D}_{n,(a,b)}\cap U_{q_0} = f^{-1}(\mathbb{R}\cdot Id)\cap U_{q_0}$ is a submanifold of codimension $2$ for $U_{q_0}$ small enough, which is a set of meagre codimension $2$.\\
When $n>2$ consider the set $E^\prime(\mathbb{R}^n)\subset End(\mathbb{R}^n)$ of all $n\times n$ matrices $L$ with entries $L_{1,1}$, $L_{1,2}$, $L_{2,1}$, and $L_{2,2}$ equal to zero, which is a space of dimension $n^2-4$. We define the differentiable map $\Bar{f}:E^\prime (\mathbb{R}^n)\times U_{q_0} \to End(\mathbb{R}^n)$ by $\Bar{f}(L,q)=L+f(q)$ which is transverse to $\mathbb{R}\cdot Id\subset End(\mathbb{R}^n)$ plus some additional direction $\langle r\rangle$ at $(0,q_0)$ by the SAH$2$ condition and the fact that the image of $f$ has at most dimension $2$. Therefore $\Bar{f}^{-1}(\mathbb{R}\cdot Id \oplus \langle r\rangle)\subset L^\prime (\mathbb{R}^n)\times U_{q_0}$ is a submanifold of codimension $n^2-2$ in a neighbourhood of $(0,q_0)$. The intersection of this submanifold with $U_{q_0}$ contains $\mathcal{D}_{n,(a,b)}\cap U_{q_0}$, and Lemma \ref{lem:codimension intersection} implies that this intersection has meagre codimension $(n^2-2)-(n^2-4)=2$ in $U_{q_0}$. This concludes the proof that $\mathcal{D}\subset X$ is a set of meagre codimension $2$.
\end{proof}

\begin{rmk}
\begin{enumerate}
    \item One can easily extend Theorem \ref{theorem: meagre codimension criterion} to complex Hilbert spaces.
    \item A reading of Teytel's proof makes clear that it is not the condition SAH$2$ that really matters, but the number $k$ of linearly independent directions of in $Im(A^\prime)$ which are transverse to $\mathbb{R}\cdot Id$ (and whose span does not contain $Id$). This number $k$ directly translates into non-simplicity of the spectrum being a meagre codimension $k$ property.
    \item We want to point out that a different approach to adapting Teytel's theorem to the case of differentiably varying inner products on the same Banach space would be to try to locally isometrically trivialize the bundle $(H\to \mathcal{X})$, consisting of fibers $H_q=(\mathcal{H},\langle-,-\rangle_{q})$ over the basepoint $q$. Supposing $\mathcal{H}$ is separable, an idea for doing this would be to fix a countable orthonormal basis (with respect to any inner product) and then apply the Gram-Schmidt procedure with respect to all $\langle-,-\rangle_q$ and hope that the resulting local frame for $H$ is smooth. Then one could consider the family $A(q)$ in this local bundle chart and try to apply Teytel's Theorem directly. Actually applying this idea to a concrete family of operators however is much more difficult than our approach as one would have to compute the Fr\'echet derivative of this transformed family of operators which of course depends on the chart one chose, so this does not seem to be useful in practice.
\end{enumerate}
\end{rmk}

\section{Main Results}
\label{section:main result}

\subsection{The curl operator has simple spectrum in the complement of a subset of meagre codimension $2$}
\label{subsection: the coclosed spectrum being simple is a meagre codimension 2 property}
We now specialize to the case $\mathcal{X}=\mathcal{G}^k$, the set of $\mathcal{C}^k$ Riemannian metrics on a closed oriented $3$-manifold $M$, endowed with the $\mathcal{C}^k$ topology. This constitutes a separable smooth Banach manifold, its tangent space $T_g \mathcal{G}^k$ at a point $g$ can be identified with $\mathcal{S}^k(M)$, the set of symmetric $(0,2)$ tensor fields of differentiability class $k$ on $M$.\par
Consider the curl operator $A_g:=\ast_g d$, the square root of $\Delta_g$ on the coexact $1$-forms. Our strategy is to show that Theorem \ref{theorem: meagre codimension criterion} applies to the curl operator.\\
The curl operator is well known to be symmetric with respect to the inner product $\langle \alpha,\beta\rangle_g=\int_M g(\alpha,\beta)\, d\mu_g$ and it admits a self-adjoint extension to a densely defined subspace of $\mathcal{H}=L^2(\Omega^1(M))$. In particular all of its eigenvalues are real and all of its eigenforms are eigenforms of $\Delta_g$ and so they are smooth $1$-forms.\\
Moreover its eigenvalues only accumulate at infinity and its eigenspaces except for the eigenvalue $0$ are finite dimensional. In order to deal with the infinite dimensional kernel of the curl operator $A_g$ we restrict it to the subspace $\mathcal{H}_g=ker(d)^{\perp_g}$. All of these Hilbert spaces are equipped with the induced inner products $\langle-,-\rangle_g$ mentioned above and can be identified with $\mathcal{H}/ker(d)$ as Banach spaces.\par
We thus see that the curl operator is of the type considered in Theorem \ref{theorem: meagre codimension criterion} and we are left with analyzing its derivative. This derivative was computed in \cite{Enciso2012}.

\begin{lem}{\cite[Lemma 2.1.]{Enciso2012}}
    Let $u$ be an eigenform of $A_g=\ast_g d$ of eigenvalue $\lambda$. Then the variation of $A_g$ in direction $h$ is given by 
    \begin{align*}
        A^\prime_g\,[h]\,u=\lambda\,h(\sharp_{g} u, -)-\frac{\lambda}{2}\,tr_g (h)\,u
    \end{align*}
where $\sharp_g$ is the canonical isomorphism of the tangent and the cotangent bundle induced by the metric $g$.
\end{lem}

\begin{rmk}
    We can see that the derivative $A^\prime_g$ of $A_g$ maps to the symmetric operators on $E(g)$. Therefore the best we can hope for is that the subspace of metrics with non-simple spectrum constitute a set of meagre codimension two. The same holds true for $\Delta_g$ restricted to the coexact spectrum as we will soon see.
\end{rmk}

Now let $\lambda$ be an eigenvalue of multiplicity $2$ and $v_1$ and $v_2$ an orthonormal basis for the associated eigenspace $E(g)$ of $A_g$. Our approach is to make special choices of $h$ and show that $A^\prime_g$ applied to these $h$ already spans a $2$-dimensional vector space transverse to $\mathbb{R}\cdot Id \subset \mathcal{L}(E(g))$. It turns out that the symmetric tensor product of the basis vectors, denoted by $v_i\odot v_j$, leads to particularly nice expressions, and so these will be our $h$.\par

Using $v_i\odot v_j(\sharp_g v_k)=\frac{1}{2}(g(v_i,v_k)v_j+g(v_j,v_k)v_i)$ and $tr_g(v_i\odot v_j)=g(v_i,v_j)$ we compute 
\begin{align*}
 &\langle A^{\prime}[v_i\odot v_j]v_k,v_l\rangle=\\
 &\frac{\lambda}{2}\,\int_M \bigg(g(v_i,v_k)\,g(v_j,v_l)+g(v_j,v_k)\,g(v_i,v_l)-g(v_i,v_j)\,g(v_k,v_l)\bigg)\,d\mu_g   
\end{align*}
We represent the linear maps $A^{\prime}[v_1\odot v_1]$ and $A^{\prime}[v_1\odot v_2]$ in this basis and obtain
\begin{align*}
 A^{\prime}[v_1\odot v_1]=  \frac{\lambda}{2} \begin{pmatrix}
        \int_M \lVert v_1\rVert^4\,d\mu_g &\int_M \lVert v_1\rVert^2\,g(v_1,v_2)\,d\mu_g\\
        \int_M \lVert v_1\rVert^2\,g(v_1,v_2)\,d\mu_g&2\int_M g(v_1,v_2)^2\,d\mu_g-\int_M \lVert v_1\rVert^2\,\lVert v_2\rVert^2\,d\mu_g
    \end{pmatrix}
\end{align*}
\begin{align*}
    A^{\prime}[v_1\odot v_2]=\frac{\lambda}{2}\,
\begin{pmatrix}
    \int_M \lVert v_1\rVert ^2\,g(v_1,v_2)\,d\mu_g & \int_M \lVert v_1\rVert^2 \,\lVert v_2\rVert^2\,d\mu_g \\
    \int_M \lVert v_1\rVert^2 \lVert v_2\rVert^2\,d\mu_g &\int_M \lVert v_2\rVert^2\,g(v_1,v_2)\,d\mu_g
\end{pmatrix}
\end{align*}

\begin{lem}
    The matrices $Id$, $A^{\prime}[v_1\odot v_1]$, and $A^{\prime}[v_1\odot v_2]$ fail to span $\mathcal{L}(E(g))$ exactly when 
\small
    \begin{align*}
        \int_M\lVert v_1\rVert^4\,d\mu_g&= \bigg(2\int_M g(v_1,v_2)^2\,d\mu_g-\int_M \lVert v_1\rVert^2 \lVert v_2\rVert ^2\,d\mu_g\bigg)\\
        &+\frac{\bigg(\int_M \lVert v_1 \rVert^2 g(v_1,v_2)\,d\mu_g\bigg)^2-\bigg(\int_M \lVert v_1 \rVert^2 g(v_1,v_2)\,d\mu_g \bigg)\,\bigg(\int_M \lVert v_2 \rVert^2 g(v_1,v_2)\,d\mu_g\bigg)}{\int_M \lVert v_1\rVert^2 \lVert v_2\rVert^2\,d\mu_g}
    \end{align*}
\normalsize
\end{lem}
\begin{proof}
    Identifying $\mathcal{L}(E(g))$ with $\mathbb{R}^3$ in the obvious way we can form the matrix built from the images of $Id$, $A^{\prime}[v_1\odot v_1]$, and $A^{\prime}[v_1\odot v_2]$ under this identification
\begin{align*}
    \begin{pmatrix}
        1&0&1\\
        \int_M \lVert v_1 \rVert^2 g(v_1,v_2)\,d\mu_g&\int_M \lVert v_1 \rVert^2 \lVert v_2\rVert^2\,d\mu_g &\int_M \lVert v_2 \rVert^2 g(v_1,v_2)\,d\mu_g\\
        \int_M \lVert v_1 \rVert^4 \,d\mu_g & \int_M \lVert v_1 \rVert^2 g(v_1,v_2)\,d\mu_g& 2\int_M g(v_1,v_2)^2\,d\mu_g-\int_M \lVert v_1\rVert^2 \lVert v_2\rVert^2\,d\mu_g
    \end{pmatrix}
\end{align*}
The vanishing of its determinant gives the condition
\begin{align*}
    &\bigg(\int_M \lVert v_1 \rVert^2 \lVert v_2\rVert^2\,d\mu_g \bigg)\bigg(2\int_M g(v_1,v_2)^2\,d\mu_g-\int_M \lVert v_1\rVert^2 \lVert v_2\rVert^2\,d\mu_g\bigg)\\
    &+\bigg(\int_M \lVert v_1\rVert^2 g(v_1,v_2)\,d\mu_g\bigg)^2 
    -\bigg(\int_M \lVert v_1\rVert^2 \lVert v_2\rVert^2\,d\mu_g\bigg)\bigg(\int_M \lVert v_1\rVert^4 \,d\mu_g\bigg)\\
    &-\bigg(\int_M \lVert v_1\rVert^2 g(v_1,v_2)\,d\mu_g\bigg)\bigg(\int_M \lVert v_2\rVert^2 g(v_1,v_2)\,d\mu_g\bigg)=0
\end{align*}
Noting that eigenforms of the curl operator are in particular eigenforms of the Hodge Laplacian, we see that they satisfy a unique continuation property (\cite{Aronszajn1957},\cite{Kazdan1988}), i.e. they vanish identically if they vanish on an open subset. It immediately follows that $\int_M \lVert v_1\rVert^2\lVert v_2\rVert^2\,d\mu_g>0$, and so we are done. 
\end{proof}

In order to deal with this degenerate case, we introduce $$\Tilde{h}_a=v_1\odot v_1+a\,tr_g[v_1\odot v_1]\, g$$ We compute once more
\begin{align*}
    A^\prime [\Tilde{h}_a]= \frac{\lambda}{2}
\begin{pmatrix}
    (1-a)\int_M \lVert v_1\rVert^4\,d\mu_g & (1-a) \int_M\lVert v_1\rVert ^2 g(v_1,v_2)\,d\mu_g\\
    (1-a)\int_M \lVert v_1\rVert^2 g(v_1,v_2)\,d\mu_g & 2\int_M g(v_1,v_2)^2-(1+a)\int_M \lVert v_1\rVert^2\lVert v_2\rVert^2\,d\mu_g
\end{pmatrix}
\end{align*}

\begin{lem}
\label{lemma: transversality of curl operator}
    $Id$, $A^{\prime}[\Tilde{h}_a]$, and $A^{\prime}[v_1\odot v_2]$ span $\mathcal{L}(E(g))$ for some choice of $a$.
\end{lem}
\begin{proof}
    Setting the determinant associated to these three vectors to zero and solving for $\int_M \lVert v_1\rVert^4\,d\mu_g$ as in the preceding Lemma we get

\begin{align*}
&\int_M\lVert v_1\rVert^4\,d\mu_g=\bigg(\frac{2}{1-a}\int_M g(v_1,v_2)^2\,d\mu_g-\frac{1+a}{1-a}\int_M \lVert v_1\rVert^2 \lVert v_2\rVert ^2\,d\mu_g\bigg)+\\
&\frac{\bigg(\int_M \lVert v_1 \rVert^2 g(v_1,v_2)\,d\mu_g\bigg)^2-\bigg(\int_M \lVert v_1 \rVert^2 g(v_1,v_2)\,d\mu_g\bigg)\,\bigg(\int_M \lVert v_2 \rVert^2 g(v_1,v_2)\,d\mu_g\bigg)}{\int_M \lVert v_1\rVert^2 \lVert v_2\rVert^2\,d\mu_g}\\
& =V_a
\end{align*}
Now suppose that in fact all of these determinants vanish, so the RHS of the above equation does not depend on the choice of $a\in \mathbb{R}\setminus \{1\}$. Then in particular, $V_a=V_0$, and so we get 
\begin{align*}
    &(1-a)\bigg(2\int_M g(v_1,v_2)^2\,d\mu_g-\int_M \lVert v_1\rVert^2\lVert v_2\rVert^2\,d\mu_g\bigg)\\
    =&\bigg(2\int_M g(v_1,v_2)^2\,d\mu_g-(1+a)\int_M \lVert v_1\rVert^2\lVert v_2\rVert^2\,d\mu_g\bigg)
\end{align*}
which is equivalent to 
\begin{align*}
    \int_M \big(g(v_1,v_2)^2-\lVert v_1\rVert^2\lVert v_2\rVert^2\big)\,d\mu_g=0
\end{align*}
meaning that $v_1\parallel v_2$ for every point in $M$. We will now show that this is impossible.
For this note that the complements of the zero set of $v_1$ and $v_2$ are open and dense by the unique continuation property. This means that there exists an open set $U\subset M$ in the complement of these zero sets and a smooth function $s$ such that $v_1=s\,v_2$ on $U$. For constant $s$ we immediately get a contradiction since $\langle v_1,v_2\rangle_g=0$, so $s$ is some nonconstant function. Plugging this into the eigenform-equation for the curl operator yields
\begin{align*}
    \lambda v_1=A_g v_1=A_g (s\,v_2)=s\,A_g v_2 +\ast_g (ds\wedge v_1)=\lambda\,v_1+\ast_g (ds\wedge v_1)
\end{align*}
This means that $ds\parallel v_1$, so $v_1=f\,ds$ for some smooth function $f$. Restricting to a subset $\Tilde{U}\subset U$ on which $ds$ does not vanish, we see that this leads to the following contradiction
\begin{align*}
    0 < v_1\wedge dv_1 = f\,ds\wedge d(f\,ds)=f^2\,ds\wedge ds =0
\end{align*}
\end{proof}

Now decompose $\mathcal{H}_g=\mathcal{H}_g^+\oplus \mathcal{H}_g^-$, where $\mathcal{H}_g^+$ and $\mathcal{H}_g^-$ are the subspaces of $\mathcal{H}_g$ spanned by the eigenforms of $A_g$ corresponding to positive and negative eigenvalues, respectively. This induces the natural splitting $\Delta_g=\Delta^+_g\oplus\Delta_g^-$.
An immediate consequence of the preceding Propositions and Theorem \ref{theorem: meagre codimension criterion} is the following

\begin{thm}
    The curl operator in dimension $3$ has simple spectrum unless $g$ is in a set of meagre codimension $2$. Moreover, the same is true for $\Delta_g^\pm$.
\end{thm}

\begin{rmk}
    One may expect that one could simply extend our proof technique to the full Hodge Laplacian $\Delta_g$ on one-forms by virtue of the following argument: given a $2$-dimensional eigenspace $E(g)$ of $\Delta_g=A_g^2$ corresponding to an eigenvalue $\lambda^2 >0$, there exists an orthonormal basis of eigenvectors of $A_g$ spanning $E(g)$. This is true because $A_g$ is an endomorphism of $E(g)$ and since $\mathcal{H}$ admits an orthonormal basis composed of eigenvectors of $A_g$. Calling this pair of eigenvectors of $A_g$ spanning $E(g)$ $v_1$ and $v_2$, we note that the $v_i$ are eigenvectors to potentially different eigenvalues, namely $\pm \lambda$.\\
Now, an easy application of the product rule and the symmetry of the operator $A_g$ with respect to $\langle-,-\rangle_g$ yields
\begin{align*}
    \langle (A_g^2)^\prime v_l,v_k\rangle_{g}=\langle A_g^\prime A_g+A_gA^\prime_g v_l,v_k\rangle_{g}=2\,\lambda\,\langle A_g^\prime v_l,v_k\rangle_{q_0}  
\end{align*}
Now if $\lambda_i=-\lambda_j$, we observe that $\Delta_g^{\prime}[v_1\odot v_2]=0$, and $\Delta_g^{\prime}[v_i\odot v_i]=2\,\lambda\,A_g^{\prime}[v_i\odot v_i]$. While this means that we cannot repeat the arguments of the previous two Lemmas for the Hodge Laplacian on coexact one-forms, we do reproduce a result of Enciso and Peralta-Salas \cite{Enciso2012}: It is not difficult to check that $A_g^{\prime}[v_1\odot v_1]$ and $A_g^{\prime}[v_2\odot v_2]$ are always nonzero and that they are linearly dependent iff $g(v_1,v_2)=0$ and $\lVert v_1\rVert=\lVert v_2\rVert$ pointwise, in which case $A_g^{\prime}[v_1\odot v_1]$ is transverse to $\mathbb{R}\cdot Id$. This easily implies that there always exists at least one $h$ such that $\Delta_g^{\prime}[h]$ is transverse to $\mathbb{R}\cdot Id$. By Lemma \ref{lemma:meagre}, we conclude that the set of Riemannian metrics for which the spectrum of the Hodge Laplacian on coexact one-forms has non-simple spectrum is meagre.\par
The proof of Corollary \ref{cor: meagre one but not two} will show that this is as far as we can go for the Hodge Laplacian on coexact $1$-forms.
\end{rmk}

\subsection{The 1-Parameter Family Arguments}
\label{subsection: the one parameter family arguments}
To prove Theorem \ref{theorem:main} we pick a new parameter space. We fix two Riemannian metrics $g_0$ and $g_1$ and define $W^k:=\{w\in C^k([0,1],\mathcal{G}^\ell):w(0)=g_0, \,w(1)=g_1\}$, where $2\leq \ell <\infty$. We use $\mathcal{X}=W^k\times (0,1)$ and define the family of operators $A^{\ast}_{(w,t)}=A_{w(t)}$. Having made these definitions we are ready to prove Theorem \ref{theorem:main}.

\begin{proof}[Proof of Theorem \ref{theorem:main}]
The plan is to once again use Theorem \ref{theorem: meagre codimension criterion}. To this end we note that $(A^\ast)^\prime_{(w,t)}[(u,s)]=A^\prime_{w(t)}[w^\prime(t)s+u(t)]$. Thus given a parameter value $(w_0,t_0)$ for which an eigenspace of $A^\ast$ is $2$-dimensional, we evaluate the derivative of $A^\ast$ at $(w_0,t_0)$ and apply it to $(u=w_0,s=0)$. This shows that $Im((A^\ast)^\prime_{(w_0,t_0)})\supseteq Im(A^\prime_{w_0(t_0)})$, and so we may conclude that the non-simple set $N$ has meagre codimension $2$ in $W^k\times (0,1)$. Lemma \ref{lemma:codimension} tells us that $\pi:(w,t)\mapsto w$ maps $N$ to a set of meagre codimension $1$ in $W^k$, which is a meagre set by Lemma \ref{lemma:meagre}. Since the complement of $\pi(N)$ in $W^k$ are precisely those $1$-parameter families which have simple coexact spectrum for all values of $t$, we are done.
\end{proof}

We can also prove, however, that an analogous statement for the nonzero spectrum of the Hodge Laplacian is false. As mentioned in the introduction, generic simplicity of the Hodge Laplacian in dimension $3$ fails in two different ways: eigenvalues corresponding to positive and negative eigenvalues of the curl operator may cross, as may exact and coexact eigenvalues. We will need the following Lemmas in the proof of these two statements.

\begin{lem}
\label{lem: continuous descent}
    Let $T:X\to Y$ be a bounded operator between Banach spaces $X$ and $Y$, and let furthermore $Z\subset X$, $W\subset Y$ be closed subspaces. Suppose $T$ vanishes on $Z$. Then 
    \begin{align*}
        \lVert \Tilde{T} \rVert_{X/Z\to Y/W}\leq \lVert T \rVert_{X\to Y}
    \end{align*}
    where $\Tilde{T}$ is the projection of $T$ to the quotient $X/Z$.
\end{lem}
\begin{proof}
We compute
\begin{align*}
    \lVert \Tilde{T} (u+Z) \rVert_{Y/W} &=\inf_{w\in W} \lVert \Tilde{T}(u+Z)+w \rVert_{Y}\leq \lVert \Tilde{T}(u+Z) \rVert_{Y}=\lVert T(u)\rVert_Y \\
    & \leq\lVert T\rVert_{X\to Y} \cdot\inf_{v\in \{u+Z\}} \lVert v\rVert_{X}=\lVert T\rVert_{X\to Y} \cdot\lVert u\rVert_{X/Z}
\end{align*}
and the result follows.
\end{proof}

\begin{lem}
\label{lem: nonvanishing of eigenvalues along one parameter families}
Let $(g_t)_{t\in [0,1]}$ be a continuous $1$-parameter family of Riemannian metrics. Then the first positive (or negative) eigenvalue $\lambda_1(t)$ is bounded away from $0$. 
\end{lem}
\begin{proof}
The curl operator $A_g:H^1 (M)\to L^2 (M)$ is bounded and descends to a bounded operator between the Banach spaces $\Tilde{A}_g: H^1(M)/ker(d)\to L^2 (M)/ker(d)$. The continuous dependence of the operator norm on $g$ factors to the quotient by Lemma \ref{lem: continuous descent}. Calling $H^1(M)/ker(d)=X$ and $L^2(M)/ker(d)=Y$ we know that for every metric $g$, the lowest (positive or negative) eigenvalue is greater than zero and so
\begin{align*}
    \lVert A_g u\rVert_Y \geq C_g \lVert u\rVert_X
\end{align*}
Fix a Riemannian metric $g_0$. Then there exists a neighbourhood $U$ of $g_0$ such that $\lVert A_g u\rVert_Y > \frac{C_{g_0}}{2}\lVert u\rVert_X$.\\
Indeed, choose $U$ so that $\lVert A_g-A_{g_0}\rVert_{X\to Y}< \frac{C_{g_0}}{2}$. Then
\begin{align*}
    \lVert A_g u\rVert_Y &\geq \lVert A_{g_0} u\rVert_Y - \lVert (A_g-A_{g_0})u\rVert_Y\\
    &\geq C_{g_0}\lVert u\rVert_X - \lVert A_g-A_{g_0}\rVert_{X\to Y}\cdot \lVert u\rVert_Y\\
    & \geq \frac{C_{g_0}}{2}\lVert u\rVert_{X}
\end{align*}
Now suppose that the family $\lambda_1 (t)$ goes to zero as $t\to t_0$. Then $C_{g_0}>0$, but $C_{g_t}$ goes to $0$ as $t$ goes to $t_0$, which is a contradiction to the above bound.
\end{proof}

It should be noted that the proof of Lemma \ref{lem: nonvanishing of eigenvalues along one parameter families} uses the fact that all $A_g$ have the same kernel. With these preparatory results out of the way we can prove collisions of the exact and coexact spectrum as well as of the positive and negative curl spectrum cannot in general be avoided along $1$-parameter families of metrics. The proof is based on the explicit descriptions of a deformation of the round metric on $S^3$ along a particular $1$-parameter family of metrics in \cite{Tanno1979LaplacianFunctions} and \cite{Tanno1983}.

\begin{proposition}\label{proposition:example of nonsimple families}
   There exist Riemannian metrics $g_0$ and $g_1$ on $\mathbb{S}^3$ such that the Hodge Laplacian has a crossing of an exact and a coexact eigenvalue along any continuous curve $g(t)$ with $g(0)=g_0$ and $g(1)=g_1$. Similarly, there exist Riemannian metrics $g_0^\prime$ and $g_1^\prime$ on $\mathbb S^3$ such that the Hodge Laplacian has a crossing of a positive and a negative curl eigenvalue along any continuous curve $g(t)$ with $g(0)=g_0^\prime$ and $g(1)=g_1^\prime$.
\end{proposition}

\begin{proof}
The idea of the proof of both statements is the same. We introduce a combinatorial invariant associated to the spectrum of the Hodge Laplacian for some metric $g$ which can only change along $1$-parameter families of metrics if a closed and a coclosed eigenvalue of the Hodge Laplacian cross (or, for the second statement, if a positive and a negative curl eigenvalue cross). Then we exhibit two metrics for which this invariant is different, thus yielding the result. We will start with the statement on exact and coexact eigenvalues.

Given a Riemannian metric $g$ so that $\Delta_g$ has simple spectrum, we colour eigenvalues red or blue depending on whether they come from an exact or a coexact eigenvalue of the Hodge Laplacian, respectively. See Figure \ref{figure:mui} for an illustration of this. Clearly the corresponding sequence of red and blue dots (counted with multiplicity) on the positive half line does not change along $1$-parameter families of metrics unless there is a crossing of a red and a blue dot. We point out that due to \cite[Theorem 1.1]{Enciso2012}, this combinatorial ordering is well defined for generic Riemannian metrics as no multiple eigenvalues of the Hodge Laplacian occur.\par
Crucially, Lemma \ref{lem: nonvanishing of eigenvalues along one parameter families} implies that (nonzero) eigenvalues of the Hodge Laplacian never go to $0$ along $1$-parameter families of Riemannian metrics.\\

\begin{figure}[h]
\centering
\includegraphics[width=0.6\textwidth]{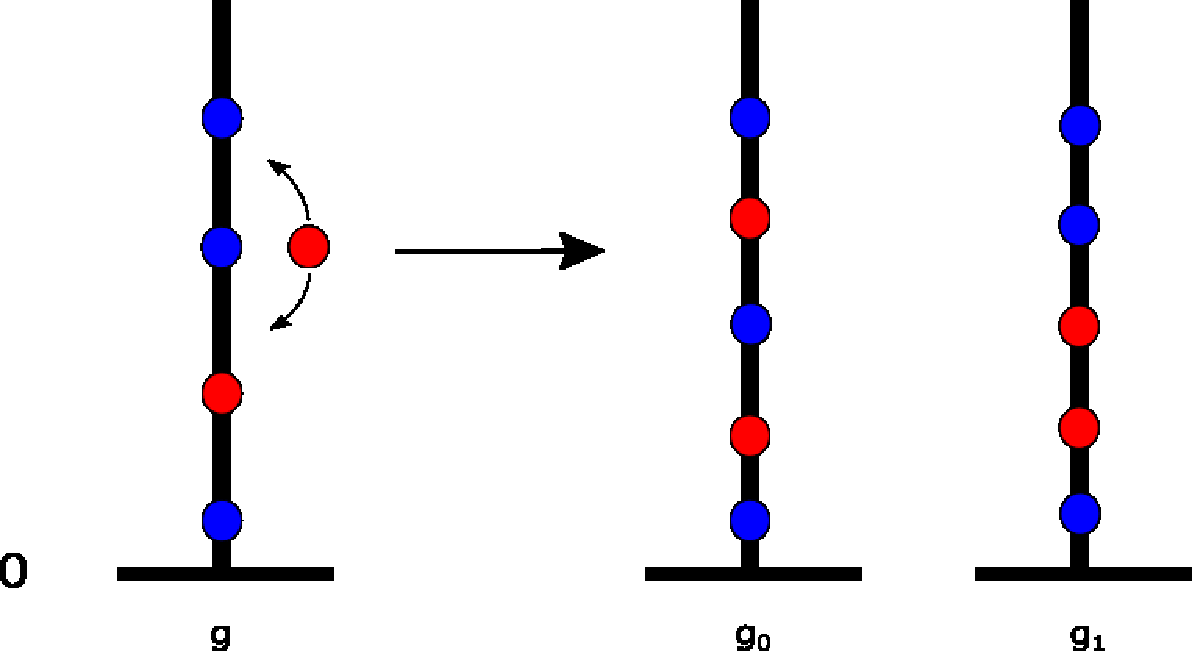}
\caption{Changing the combinatorial structure of the spectrum.}
\label{figure:mui}
\end{figure}

Consequently, if we find two Riemannian metrics $g_0$ and $g_1$ for which the Hodge Laplacian has simple spectrum but whose spectrum defines a different sequence of colored points, we have proven the first part of Proposition \ref{proposition:example of nonsimple families}, see Figure \ref{figure:mui} for an illustration of this process.\\

The main ingredient for constructing two metrics whose associated spectra of the Hodge Laplacian define different sequences of coloured points is a particular deformation of the round metric $g_{st}$ on odd dimensional spheres considered by Tanno in \cite{Tanno1979LaplacianFunctions} and \cite{Tanno1983}. We will specialize their discussion to the $3$-dimensional setting. Let $\alpha$ be a Hopf field on $S^3$, then consider the $1$-parameter family of Riemannian metrics given by
\begin{align*}
    g_t = t^{-1} g_{st} +(t^{-2}-t^{-1})\,\alpha\otimes\alpha 
\end{align*}
Using the fact that the Laplacians $\Delta_{g_t}$ and $\Delta_{g_{st}}$ on functions commute, Tanno succeeded in computing the full effect of this deformation on the spectrum of the Laplacian on functions in \cite[Lemma 4.1. and Proposition 4.2.]{Tanno1979LaplacianFunctions}. In particular, they proved that the first eigenvalue of the Laplacian on functions is given by
\begin{align*}
    \lambda_1^0 (t) = \begin{cases}
    2t+t^{-2} & \text{if } t^{-3} \leq 6\\
        8t                                  & \text{if } t^{-3} \geq 6
    \end{cases} 
\end{align*}
Note that at $t=1$ (the round metric), the first eigenvalue of the Laplacian on functions is $3$, wheres the square of the first curl eigenvalue is $4$. Thus, the combinatorial structure of the first few eigenvalues of the Hodge Laplacian at $g_1 = g_{st}$ is that the first four eigenvalues (counted with multiplicity) belong to the exact spectrum, and the next $6$ eigenvalues (counted with multiplicity) belong to the coexact spectrum. While the Hodge Laplacian restricted to the coexact spectrum does not have the nice commutation relation mentioned above, Tanno proved in \cite[Theorem 4.1.]{Tanno1983} that there is a coexact eigenvalue of the Hodge Laplacian which goes below $\lambda_1^0 (t)$, the crossing happening at $t=\big(\frac{1}{4}+\frac{\sqrt{5}}{4}\big)^{\frac{1}{3}}$, thus changing the combinatorial structure of the spectrum of the Hodge Laplacian to one where the first eigenvalue is a coexact one. This concludes the proof of the first claim.

For the second claim we consider the same family of metrics but concentrate on the deformation of the $6$-dimensional eigenspace of the Hodge Laplacian associated to the lowest coexact eigenvalue $\lambda=4$ with respect to the round metric $g_1=g_{st}$. This eigenspace is spanned by the three Hopf fields (the solutions of the curl eigenvalue problem for the eigenvalue $\lambda = 2$ and the three anti-Hopf fields (the solutions of the curl eigenvalue problem for the eigenvalue $\lambda=-2$). The anti-Hopf fields are obtained from the Hopf fields via an application of the orientation reversing involution $(x_1,y_1,x_2,y_2)\mapsto (x_1,y_1,x_2,-y_2)$ of $\mathbb R^4$ restricted to $\mathbb{S}^3$, see for example \cite[Remark 6]{PeraltaSlobodeanu2021}.
Tanno proves \cite[Theorem 4.1]{Tanno1983} that, along the family of metrics $g_t$, this eigenspace splits into three eigenspaces associated to the three eigenvalues:
\begin{enumerate}
    \item $4t^2$
    \item $2t^{-2}$
    \item $8t + 2t^4 - 2t(t^6+8t^3)^{\frac{1}{2}} $
\end{enumerate}
The first of these corresponds to the Hopf field $\alpha$, the second one corresponds to the remaining two Hopf fields, and the third one corresponds to the three anti-Hopf fields and therefore the first negative eigenvalue of the curl operator $g_t$ for $t$ close enough to $1$. Thus we have that along the transition from $t=1-\varepsilon$ to $t=1+\varepsilon$ the combinatorial structure of the coexact spectrum changes from one where the first eigenvalue is positive, followed by three negative eigenvalues (counted with multiplicity) to one where the first two eigenvalues (counted with multiplicity) are positive followed by three negative ones. Breaking up these higher multiplicities with a generic perturbation (which, as mentioned previously, we can do according to \cite[Theorem 1.1.]{Enciso2012}) completes the proof.
\end{proof}

We now get Theorem \ref{theorem: Hodge Laplacian 1-parameter genericity} and Corollary \ref{cor: meagre one but not two} as easy consequences of everything that has been discussed in Section \ref{section:main result}.

\begin{proof}[Proof of Theorem \ref{theorem: Hodge Laplacian 1-parameter genericity}]
    Proposition \ref{proposition:example of nonsimple families} implies that the spectrum of the Hodge Laplacian is not simple along generic $1$-parameter families of Riemannian metrics. The result by Uhlenbeck on the simplicity of the Laplace-Beltrami operator along $1$-parameter families of Riemannian metrics, Lemma \ref{lemma: transversality of curl operator} and the proof of Theorem \ref{theorem:main} allow us to conclude simplicity of the positive and negative coexact spectrum along generic $1$-parameter families.
\end{proof}

An immediate consequence of Proposition \ref{proposition:example of nonsimple families} is the proof of Corollary \ref{cor: meagre one but not two}.

\begin{proof}[Proof of Corollary \ref{cor: meagre one but not two}]
    By Lemma \ref{lemma:meagre}, a set has meagre codimension $1$ iff it is meagre. The fact that the non-simple set of Riemannian metrics for which the Hodge Laplacian in dimension $3$ does not have simple nonzero spectrum has meagre codimension $1$ was proven in \cite{Enciso2012}.\\
    Now by Proposition \ref{proposition:example of nonsimple families}, the non-simple set cannot have meagre codimension $2$ since otherwise we would get generic simplicity for $1$-parameter families.
\end{proof}

\printbibliography

@article{Uhlenbeck1976,
   author = {K. Uhlenbeck},
   doi = {10.2307/2374041},
   issue = {4},
   journal = {American Journal of Mathematics},
   title = {Generic Properties of Eigenfunctions},
   volume = {98},
   year = {1976},
}

@article{Enciso2012,
   author = {Alberto Enciso and Daniel Peralta-Salas},
   doi = {10.1090/s0002-9947-2012-05496-1},
   issue = {8},
   journal = {Transactions of the American Mathematical Society},
   title = {Nondegeneracy of the eigenvalues of the Hodge Laplacian for generic metrics on 3-manifolds},
   volume = {364},
   year = {2012},
}

@article{Etnyre2000III,
   author = {John Etnyre and Robert Ghrist},
   doi = {10.1088/0951-7715/13/2/306},
   issue = {2},
   journal = {Nonlinearity},
   title = {Contact topology and hydrodynamics: I. Beltrami fields and the Seifert conjecture},
   volume = {13},
   year = {2000},
}

@article{Teytel1999,
   author = {Mikhail Teytel},
   doi = {10.1002/(sici)1097-0312(199908)52:8<917::aid-cpa1>3.3.co;2-j},
   issue = {8},
   journal = {Communications on Pure and Applied Mathematics},
   title = {How rare are multiple eigenvalues?},
   volume = {52},
   year = {1999},
}

@article{Verdière1988,
   author = {Y. Colin de Verdière},
   doi = {10.1007/BF02566761},
   issue = {1},
   journal = {Commentarii Mathematici Helvetici},
   title = {Sur une hypothèse de transversalité d'Arnold},
   volume = {63},
   year = {1988},
}

@article{Lipnowski2018,
   author = {Michael Lipnowski and Mark Stern},
   doi = {10.1007/s00039-018-0471-x},
   issue = {6},
   journal = {Geometric and Functional Analysis},
   title = {Geometry of the Smallest 1-form Laplacian Eigenvalue on Hyperbolic Manifolds},
   volume = {28},
   year = {2018},
}

@article{Etnyre2012,
   author = {John B. Etnyre and Rafal Komendarczyk and Patrick Massot},
   doi = {10.1007/s00222-011-0355-2},
   issue = {3},
   journal = {Inventiones Mathematicae},
   title = {Tightness in contact metric 3-manifolds},
   volume = {188},
   year = {2012},
}

@article{Lin2021,
   author = {Francesco Lin and Michael Lipnowski},
   doi = {10.1090/jams/982},
   issue = {1},
   journal = {Journal of the American Mathematical Society},
   title = {The Seiberg-Witten equations and the length spectrum of hyperbolic three-manifolds},
   volume = {35},
   year = {2021},
}

@article {Millman1980,
    AUTHOR = {Millman, Richard S.},
     TITLE = {Remarks on spectrum of the {L}aplace-{B}eltrami operator in
              the middle dimensions},
   JOURNAL = {Tensor (N.S.)},
  FJOURNAL = {The Tensor Society. Tensor. New Series},
    VOLUME = {34},
      YEAR = {1980},
    NUMBER = {1},
     PAGES = {94--96},
   MRCLASS = {58G25},
  MRNUMBER = {570572},
MRREVIEWER = {Akira Asada},
}

@article {GierHislop2016,
    AUTHOR = {Gier, Megan E. and Hislop, Peter D.},
     TITLE = {The multiplicity of eigenvalues of the {H}odge {L}aplacian on
              5-dimensional compact manifolds},
   JOURNAL = {J. Geom. Anal.},
  FJOURNAL = {Journal of Geometric Analysis},
    VOLUME = {26},
      YEAR = {2016},
    NUMBER = {4},
     PAGES = {3176--3193},
   MRCLASS = {58J50 (35R01 47A55)},
  MRNUMBER = {3544958},
MRREVIEWER = {Akira Asada},
       DOI = {10.1007/s12220-015-9666-7},
       URL = {https://doi.org/10.1007/s12220-015-9666-7},
}

@article {Tanno1983,
    AUTHOR = {Tanno, Shukichi},
     TITLE = {{G}eometric expressions of eigen $1$-forms of the {L}aplacian on {S}pheres},
   JOURNAL = {Spectra of Riemannian manifolds},
  FJOURNAL = {Spectra of Riemannian manifolds},
    VOLUME = {},
      YEAR = {1983},
    NUMBER = {},
     PAGES = {115--128},
   MRCLASS = {},
  MRNUMBER = {},
MRREVIEWER = {},
       DOI = {},
       URL = {},
}

@article {RieszNagy1955,
    AUTHOR = {Riesz, F and Sz.-Nagy, B.},
     TITLE = {{F}unctional {A}nalysis},
   JOURNAL = {Dover Books on Advanced Mathematics},
  FJOURNAL = {Dover Pulbications},
    VOLUME = {},
      YEAR = {1955},
    NUMBER = {},
     PAGES = {},
   MRCLASS = {},
  MRNUMBER = {},
MRREVIEWER = {},
       DOI = {},
       URL = {},
}

@incollection {CarMirPerPres2019,
    AUTHOR = {Cardona, Robert and Miranda, Eva and Peralta-Salas, Daniel and
              Presas, Francisco},
     TITLE = {Reeb embeddings and universality of {E}uler flows},
 BOOKTITLE = {Extended abstracts {GEOMVAP} 2019---geometry, topology,
              algebra, and applications; women in geometry and topology},
    SERIES = {Trends Math. Res. Perspect. CRM Barc.},
    VOLUME = {15},
     PAGES = {115--120},
 PUBLISHER = {Birkh\"{a}user/Springer, Cham},
      YEAR = {[2021] \copyright 2021},
   MRCLASS = {53E50 (35Q31 37J55)},
  MRNUMBER = {4436874},
       DOI = {10.1007/978-3-030-84800-2\_19},
       URL = {https://doi.org/10.1007/978-3-030-84800-2_19},
}

@book {Hodge1941,
    AUTHOR = {Hodge, W. V. D.},
     TITLE = {The theory and applications of harmonic integrals},
    SERIES = {Cambridge Mathematical Library},
      NOTE = {Reprint of the 1941 original,
              With a foreword by Michael Atiyah},
 PUBLISHER = {Cambridge University Press, Cambridge},
      YEAR = {1989},
     PAGES = {xiv+284},
   MRCLASS = {58A14 (01A60 01A75 32G20 53-02)},
  MRNUMBER = {1015714},
}

@article {GomesMarrocos2019,
    AUTHOR = {Gomes, Jos\'{e} N. V. and Marrocos, Marcus A. M.},
     TITLE = {On eigenvalue generic properties of the {L}aplace-{N}eumann
              operator},
   JOURNAL = {J. Geom. Phys.},
  FJOURNAL = {Journal of Geometry and Physics},
    VOLUME = {135},
      YEAR = {2019},
     PAGES = {21--31},
   MRCLASS = {47A75 (35J05 35P05 35R01 47A55 58C40)},
  MRNUMBER = {3872620},
       DOI = {10.1016/j.geomphys.2018.08.017},
       URL = {https://doi.org/10.1016/j.geomphys.2018.08.017},
}

@article {Junya2003,
    AUTHOR = {Takahashi, Junya},
     TITLE = {On the gap between the first eigenvalues of the {L}aplacian on
              functions and {$p$}-forms},
   JOURNAL = {Ann. Global Anal. Geom.},
  FJOURNAL = {Annals of Global Analysis and Geometry},
    VOLUME = {23},
      YEAR = {2003},
    NUMBER = {1},
     PAGES = {13--27},
   MRCLASS = {58J50 (35P15)},
  MRNUMBER = {1952856},
MRREVIEWER = {Colette Ann\'{e}},
       DOI = {10.1023/A:1021294732338},
       URL = {https://doi.org/10.1023/A:1021294732338},
}

@article {ColetteJunya2012,
    AUTHOR = {Ann\'{e}, Colette and Takahashi, Junya},
     TITLE = {{$p$}-spectrum and collapsing of connected sums},
   JOURNAL = {Trans. Amer. Math. Soc.},
  FJOURNAL = {Transactions of the American Mathematical Society},
    VOLUME = {364},
      YEAR = {2012},
    NUMBER = {4},
     PAGES = {1711--1735},
   MRCLASS = {58J50 (35P15)},
  MRNUMBER = {2869189},
MRREVIEWER = {Ahmad El Soufi},
       DOI = {10.1090/S0002-9947-2011-05351-1},
       URL = {https://doi.org/10.1090/S0002-9947-2011-05351-1},
}

@book {CaivoveanuRassias2001,
    AUTHOR = {Craioveanu, Mircea and Puta, Mircea and Rassias, Themistocles
              M.},
     TITLE = {Old and new aspects in spectral geometry},
    SERIES = {Mathematics and its Applications},
    VOLUME = {534},
 PUBLISHER = {Kluwer Academic Publishers, Dordrecht},
      YEAR = {2001},
     PAGES = {x+445},
   MRCLASS = {58J50 (35J25 35P05 53C20)},
  MRNUMBER = {1880186},
       DOI = {10.1007/978-94-017-2475-3},
       URL = {https://doi.org/10.1007/978-94-017-2475-3},
}

@article {Aronszajn1957,
    AUTHOR = {Aronszajn, N.},
     TITLE = {A unique continuation theorem for solutions of elliptic
              partial differential equations or inequalities of second
              order},
   JOURNAL = {J. Math. Pures Appl. (9)},
  FJOURNAL = {Journal de Math\'{e}matiques Pures et Appliqu\'{e}es. Neuvi\`eme S\'{e}rie},
    VOLUME = {36},
      YEAR = {1957},
     PAGES = {235--249},
   MRCLASS = {35.0X},
  MRNUMBER = {92067},
MRREVIEWER = {H. Bremekamp},
}

@article {Kazdan1988,
    AUTHOR = {Kazdan, Jerry L.},
     TITLE = {Unique continuation in geometry},
   JOURNAL = {Comm. Pure Appl. Math.},
  FJOURNAL = {Communications on Pure and Applied Mathematics},
    VOLUME = {41},
      YEAR = {1988},
    NUMBER = {5},
     PAGES = {667--681},
   MRCLASS = {35B60 (35J25 58G30)},
  MRNUMBER = {948075},
MRREVIEWER = {Steven George Krantz},
       DOI = {10.1002/cpa.3160410508},
       URL = {https://doi.org/10.1002/cpa.3160410508},
}

@article {PeraltaSlobodeanu2021,
    AUTHOR = {Peralta-Salas, Daniel and Slbodeanu, Radu},
     TITLE = {Contact Structures and Beltrami Fields on the Torus and the Sphere},
   JOURNAL = {Arxiv},
      YEAR = {2021},
       URL = {https://arxiv.org/pdf/2004.10185.pdf},
}

@book {GeraldTeschl2014,
    AUTHOR = {Teschl, Gerald},
     TITLE = {Mathematical methods in quantum mechanics},
    SERIES = {Graduate Studies in Mathematics},
    VOLUME = {157},
   EDITION = {Second},
      NOTE = {With applications to Schr\"{o}dinger operators},
 PUBLISHER = {American Mathematical Society, Providence, RI},
      YEAR = {2014},
     PAGES = {xiv+358},
   MRCLASS = {81-02 (47N50 81Q10 81Q12 81U10)},
  MRNUMBER = {3243083},
MRREVIEWER = {Rupert L. Frank},
       DOI = {10.1090/gsm/157},
       URL = {https://doi.org/10.1090/gsm/157},
}

@article {Tanno1979LaplacianFunctions,
    AUTHOR = {Tanno, Sh\^ukichi},
     TITLE = {The first eigenvalue of the {L}aplacian on spheres},
   JOURNAL = {Tohoku Math. J. (2)},
  FJOURNAL = {The Tohoku Mathematical Journal. Second Series},
    VOLUME = {31},
      YEAR = {1979},
    NUMBER = {2},
     PAGES = {179--185},
      ISSN = {0040-8735,2186-585X},
   MRCLASS = {58G25 (53C20)},
  MRNUMBER = {538918},
MRREVIEWER = {R.\ S.\ Millman},
       DOI = {10.2748/tmj/1178229837},
       URL = {https://doi.org/10.2748/tmj/1178229837},
}

\end{document}